\documentclass[11pt,letterpaper]{article}
\usepackage[margin=1in,marginpar=0.8in]{geometry}
\usepackage[utf8]{inputenc}
\usepackage[T1]{fontenc}
\usepackage{authblk}

\usepackage{amssymb, amsthm}
\usepackage[leqno]{amsmath}
\usepackage{tikz, url}
\usetikzlibrary{snakes}

\usepackage[shortlabels]{enumitem}

\usepackage[normalem]{ulem}
\usepackage{setspace, csquotes}
\MakeOuterQuote{"}
\onehalfspace
\usepackage{changepage}
\usepackage{algorithm}
\usepackage{algcompatible}
\usepackage{algpseudocode}


\usepackage{lscape}

\usepackage{tcolorbox}
\newtheorem{definition}{Definition}

\newtheorem{thm}{Theorem}

\newtheorem{lemma}[thm]{Lemma}

\newtheorem{claim}{Claim}

\usepackage{bbold}
\usepackage{enumitem}

\usepackage{xcolor}
\definecolor{nicelavender}{RGB}{153, 128, 250}
\definecolor{nicegreen}{RGB}{106, 176, 76}
\definecolor{random}{RGB}{60, 170, 200}
\definecolor{nicepurple}{RGB}{200, 90, 160}
\definecolor{darkpurple}{RGB}{164, 10, 106}
\definecolor{darkblue}{RGB}{0, 0, 100}
\newcommand\cece[1]{{\color{nicegreen}{#1}}}

\begin{document}

\author{Cicely Henderson$^\dagger$}
\author{Evelyne Smith-Roberge$^\flat$}
\author{Sophie Spirkl\thanks{We acknowledge the support of the Natural Sciences and Engineering Research Council of Canada (NSERC), [funding reference number RGPIN-2020-03912]. Cette recherche a \'et\'e financ\'ee par le Conseil de recherches en sciences naturelles et eng\'enie du Canada (CRSNG), [num\'ero de r\'ef\'erence RGPIN-2020-03912]. This project was funded in part by the Government of Ontario. This research was completed while Spirkl was an Alfred P. Sloan Fellow.}}
\author{Rebecca Whitman$^\sharp$}
\affil{$^{\dagger, *}$Dept.~of Combinatorics and Optimization, University of Waterloo \\ \texttt{\{c3hender, sspirkl\}@uwaterloo.ca}}
\affil{$^\flat$Dept. of Mathematics, Illinois State University \\ \texttt{esmithr@ilstu.edu}}
\affil{$^\sharp$Dept. of Mathematics,
University of California, Berkeley \\ \texttt{rebecca\_whitman@berkeley.edu}}
\title{Maximum $k$-colourable induced subgraphs in $(P_5+rK_1)$-free graphs}
\date{\today}
\maketitle

\begin{abstract}
We show that for any nonnegative integer $r$, the \textsc{Weighted Maximum List-$k$-Colourable Induced Subgraph} problem can be solved in polynomial time for input graphs that do not contain $(P_5+ rK_1)$ as an induced subgraph, and give an explicit algorithm demonstrating this. This answers a question of Agrawal et al.\ (2024). 
\end{abstract}

\section{Introduction}

Let $G$ be a graph and $\omega : V(G) \rightarrow \mathbb{R}_{+}$ be a weight function. For a subgraph $H$ of $G$, we define $\omega(H):= \sum_{v \in V(H)} \omega(v)$. We denote by $v(G)$ the number of vertices in $G$. Given a positive natural number $k$, we use $[k]$ to denote the set $\{1, 2, \dots, k\}$, and $P_k$ to denote the path with $k$ vertices. We use $G + H$ to denote the disjoint union of graphs $G$ and $H$, and $kG$ to denote the disjoint union of $k$ isomorphic copies of $G$. Given graphs $G$ and $H$, we say $G$ is \emph{$H$-free} if $G$ does not contain an induced subgraph isomorphic to $H$. For $X, Y \subseteq V(G)$, we say that $X$ is \emph{anticomplete} to $Y$ if $X \cap Y = \emptyset$ and $G$ contains no edge with one end in $X$ and the other in $Y$. 

Given a graph $G$, a \emph{list-$k$-assignment} of $G$ is a function $L: V(G) \rightarrow 2^{[k]}$ that associates to each vertex $v \in V(G)$ a list $L(v) \subseteq [k]$ of colours. Given a list assignment $L$, an $L$-colouring of $G$ is a proper colouring $\phi$ of $G$ such that $\phi(v) \in L(v)$ for all $v \in V(G)$. 

The \textsc{Weighted Maximum List-$k$-Colourable Induced Subgraph} problem (henceforth abbreviated as WML$k$CIS) is the following: 
\begin{itemize}
    \item Input: A graph $G$, list-$k$-assignment $L$ of $G$ and weight function $\omega: V(G) \rightarrow \mathbb{R}_{+}$. 
    \item Output: An $L$-colourable induced subgraph $H$ of $G$ that maximizes $\omega(H)$ over all such subgraphs of $G$. 
\end{itemize}
This generalizes the \textsc{Weighted Maximum $k$-Colourable Induced Subgraph} problem, in which we require $L(v) = [k]$ for all $v \in V(G)$. WML$k$CIS also generalizes the \textsc{List-$k$-Colouring} problem, which asks if $G$ admits an $L$-colouring.

The WML$k$CIS problem is difficult: in particular, it is NP-hard for every fixed $k$, and even NP-hard to approximate within a factor of $v(G)^\epsilon$ for some fixed $\epsilon > 0$ (per Lund and Yannakakis \cite{lund1993approximation}). It is natural, therefore, to restrict our attention to specific graph classes and try to solve the problem there. Another possible approach is to restrict our study to fixed, small values of $k$.

The $k=1$ and $k=2$ cases are particularly well-studied. When $k=1$, the problem is equivalent to the NP-hard problem of finding a maximum weight independent set (MWIS) in the induced subgraph of the input graph consisting of vertices with a non-empty list. When $k=2$ and all vertices have list $\{1, 2\}$, the problem is equivalent to the  \textsc{Odd Cycle Transversal} (OCT) problem, that is, the problem of finding a set of vertices whose deletion results in a maximum-weight graph with no odd cycles. 

Since complexity results for WML$k$CIS often seem to align with those for the OCT problem, we review what is known for this problem in $H$-free graphs. It follows from Chiarelli et al.\ \cite{chiarelli2018} that if a component of $H$ contains a cycle or a claw, then the OCT problem is NP-hard for $H$-free graphs. If no component of $H$ contains a cycle or claw, then $H$ is a \emph{linear forest}, that is, each component is a path. For $P_4$-free graphs, the OCT problem is solvable in polynomial time (per Brandstädt and Kratsch \cite{brandstadt1985}). For every integer $r \geq 1$, it is also solvable in polynomial time for $rP_2$-free graphs (as shown by Chiarelli et al.\ \cite{chiarelli2018}), and $(rK_1+P_3)$-free graphs (as shown by Dabrowski et al.\ \cite{dabrowski2020}). It was shown by Okrasa and Rz\k{a}żewski \cite{okrasa2020} to be NP-hard for $P_{13}$-free graphs, and per \cite{dabrowski2020}, it is also NP-hard for $(P_2+P_5,P_6)$-free graphs. Consequently, the problem is NP-hard for $P_t$-free graphs with $t \ge 6$. Until very recently, the only path $P$ for which the complexity of the OCT problem for $P$-free graphs was unknown was $P_5$; this final case was resolved by Agrawal, Lima, Lokshtanov, Rz\k{a}żewski, Saurabh, and Sharma  \cite{agrawal2024odd}, who gave a polynomial-time algorithm for the problem. This answered a question asked by Rz\k{a}żewski (as reported in \cite{chudnovsky2019}), and reiterated by Chudnovsky et al.\ \cite{chudnovsky2021}.

In \cite{agrawal2024odd}, the authors ask whether the WML$k$CIS problem can be solved for $P_5$-free graphs for general $k$ in the special case when $L(v) = [k]$ for every vertex $v$ of the input graph. In this paper, we answer their question in the affirmative; in fact, we show the following stronger statement (without the restriction on lists). 

 \begin{thm}\label{thm:mainthmp5}
    Let $k$ and $r$ be fixed nonnegative integers. The WML$k$CIS problem can be solved in polynomial time for all $(P_5+rK_1)$-free graphs, list-$k$-assignments, and weight functions.
\end{thm}

Couturier et al.\ \cite{couturier2015list} showed that \textsc{List-$k$-Colouring} is NP-complete in $P_4+P_2$-free graphs for all $k \geq 5$. Along with the aforementioned NP-hardness result for OCT in $P_6$-free graphs \cite{dabrowski2020}, this leaves only two possible cases when WML$k$CIS might be polynomial-time solvable: when $H$ is an induced subgraph of $rP_3$ (which is shown to be polynomial-time solvable in the recent work of Galby et al. \cite{galby2025}) and when $H$ is an induced subgraph of $P_5+rK_1$, our main result above. This completes the complexity dichotomy of WML$k$CIS for $k \ge 5$.

\begin{thm}[\cite{galby2025}, Theorem 4] \label{thm:complexitydichotomy} 
    For $k \ge 5$, the WML$k$CIS problem can be solved in polynomial time for all $H$-free graphs if and only if $H$ is an induced subgraph of $rP_3$ or $P_5 + rK_1$, for some $r \ge 1$. 
\end{thm}

We note that if one can solve the WML$k$CIS problem for a graph $G$ efficiently for every fixed $k$ and list-$k$-assignment $L$, then one can also determine: 
    \vspace{-4mm}
    \begin{itemize}\itemsep -2pt
        \item the maximum size of an independent set in $G$ (in our case, we use the algorithm of Lokshtanov et al.\ \cite{lokshtanov2014independent}, and so our results do not lead to a new algorithm for MWIS in $P_5$-free graphs); 
        \item as mentioned prior, the minimum-weight odd cycle transversal, and so Theorem \ref{thm:mainthmp5} generalizes the afore-mentioned result of Agrawal et al.\ \cite{agrawal2024odd}; and 
        \item the \textsc{List-$k$-Colouring} problem for $G$, and so Theorem \ref{thm:mainthmp5} generalizes a result of Couturier et al.\ \cite{couturier2015list} (which in turn generalized Ho\`ang et al.\ \cite{hoang2010deciding}). 
    \end{itemize}
    \vspace{-4mm}

Efficient algorithms for any of these three problems do not, however, guarantee an efficient algorithm for the WML$k$CIS problem.

In Section \ref{sec:prefatory_results} we prove several necessary lemmas about $(P_5 +rK_1)$-free graphs. Section \ref{sec:p5rk1} contains the proof of Theorem \ref{thm:mainthmp5}. 





\section{Prefatory Results}
\label{sec:prefatory_results}

We provide several prefatory results on $(P_5 + rK_1)$-free graphs. For a fixed  nonnegative integer $r$, we denote by $\mathcal{G}_r$ the class of $(P_5 + rK_1)$-free graphs. For a given graph $G$, vertex $v \in V(G)$, and vertex subset $S \subseteq V(G)$, let $N_G(v)$ denote the neighbour set of $v$ in $G$ and let $N_G(S) = \bigcup_{v \in S} N_G(v)  \setminus S$. Let $N_G\langle S\rangle = S \cup \bigcup_{v \in S} N_G(v)$ denote the closed neighbourhood of a set $S$. Where the choice of $G$ is clear, we omit the subscript. Let $G[S]$ denote the induced subgraph of $G$ with vertex set $S$. We define an \emph{ordered set} to be a set together with a fixed ordering of the elements. 

The following is an easy consequence of the polynomial-time algorithm for solving the MWIS in $P_5$-free graphs, due to Lokshtanov, Vatshelle, and Villanger \cite{lokshtanov2014independent}. 

\begin{lemma} \label{lem:stable}
    Fix $r \in \mathbb{Z}_{\ge 0}$. There is a polynomial-time algorithm for solving the MWIS problem on $(P_5 + rK_1)$-free graphs.
\end{lemma}
\begin{proof}
    If $r = 0$, we use the algorithm given by Lokshtanov, Vatshelle, and Villanger \cite{lokshtanov2014independent}.

    Otherwise, let $r \ge 1$. Let $G \in \mathcal{G}_r$ be a graph, and let $\omega: V(G) \rightarrow \mathbb{R}_{+}$ be a weight function. For $i \in \{1, \dots, r\}$, let $\mathcal{S}_i$ denote the set of all independent sets of size $i$ in $G$. For every $T \in \mathcal{S}_r$, let $T^+$ denote a maximum-weight independent set in $G \setminus N\langle T \rangle$. Let $\mathcal{S} = \{T \cup T^+ : T \in \mathcal{S}_r\} \cup \mathcal{S}_1 \cup \dots \cup \mathcal{S}_r$. 

    We claim that: 
    \vspace{-4mm}
    \begin{itemize}\itemsep -2pt
        \item $\mathcal{S}$ can be computed in polynomial time. 
        \item Every $S \in \mathcal{S}$ is an independent set in $G$. 
        \item $\mathcal{S}$ contains a maximum-weight independent set of $G$. 
    \end{itemize}
    \vspace{-4mm}
    First, note that for each $i \in \{1, \dots, r\}$, we have that $|\mathcal{S}_{i}| \leq \binom{v(G)}{{i}} \leq v(G)^{i}$. The first and second statements  immediately follow for the sets $\mathcal{S}_1, \dots, \mathcal{S}_r$.  For $T \in \mathcal{S}_r$, note that $G \setminus N\langle T \rangle$ is $P_5$-free (otherwise, adding $T$ to a copy of $P_5$ in  $G \setminus N\langle T \rangle$ gives a copy of $P_5+rK_1$ in $G$, a contradiction). Therefore for each $T \in \mathcal{S}_r$, we have that $T^+$ can be computed in polynomial time using the algorithm of \cite{lokshtanov2014independent}. Recalling again that $|\mathcal{S}_r| \leq v(G)^r$, it follows that the first statement holds for the remaining set $\{T \cup T^+ : T \in \mathcal{S}_r\}$ of $\mathcal{S}$. Since for each $T \in \mathcal{S}_r$ we have that $T^+$ and $T$ are each independent sets and $T^+$ contains no neighbours of $T$ by definition, it follows that $T \cup T^+$ is independent, and so the second statement holds. 

    It remains to prove the third statement. Let $S$ be a maximum-weight independent set in $G$. If $|S| \leq r$, then $S \in \mathcal{S}_{|S|}$, and hence $S \in \mathcal{S}$, as desired. Therefore, we may assume that $|S| > r$. Let $T \subseteq S$ with $|T| = r$. Then $T \in \mathcal{S}_r$, and moreover, $S \setminus T$ is an independent set in $G \setminus N\langle T \rangle$. It follows that $\omega(T^+) \geq \omega(S \setminus T)$, and so $\omega(T \cup T^+) \geq \omega(T) + \omega(S \setminus T) = \omega(S)$. Therefore, $\mathcal{S}$ contains an independent set of maximum weight, which establishes the third statement. 
    
    Given the three statements above, the following is an algorithm to generate a maximum-weight independent set of $G$ in polynomial time: First, we compute $\mathcal{S}$ (which can be done in polynomial time by the first statement). Next, we choose an element of $\mathcal{S}$ of maximum weight (which can be done concurrently with computing $\mathcal{S}$, by keeping track of the highest-weight element of $\mathcal{S}$ generated so far).
\end{proof}

Given a graph $G$ and an edge $uv \in E(G)$, we define $G/uv$ to be the graph obtained from $G$ by removing $u$ and $v$, and adding a new vertex $w$ with $N_{G/uv}(w) = (N(u) \cup N(v)) \setminus \{u, v\}$. The operation transforming $G$ into $G/uv$ is called \emph{edge contraction} (together with \emph{simplification,} if $u$ and $v$ have a common neighbour in $G$). Similar to \cite{gartland2021finding}, we show in Lemma \ref{lemma:P5+rk1freegraphsnice} that edge contraction preserves being $(P_5 + rK_1)$-free; in fact, we present a more general argument showing that for every linear forest $H$, edge contraction preserves being $H$-free.

\begin{lemma}
\label{lemma:P5+rk1freegraphsnice}
    Let $H$ be a linear forest. The class of $H$-free graphs is closed under edge contraction.  
\end{lemma}
\begin{proof}
    Assume for a contradiction that there exists some $H$-free graph $G$ and an edge $uv\in E(G)$ such that $G/ uv$ contains an induced copy of $H$. Call this induced copy $P$. Let $w \in V(G/uv)$ be the vertex corresponding to the contracted edge $uv$. If $w \notin V(P)$, then $P$ is an induced subgraph of $G$, contradicting that $G$ is $H$-free. Hence we may assume $w \in V(P)$. If $N(w) \cap V(P) \subseteq N(u)$, then $(V(P) \cup \{u\}) - \{w\}$ induces a copy of $H$ in $G$. Likewise, if $N(w) \cap V(P)\subseteq N(v)$, then $(V(P) \cup \{v\}) - \{w\}$ induces a copy of $H$ in $G$, a contradiction. Therefore, $w$ corresponds to a vertex of degree 2 in $P$, and furthermore, letting $w'$ and $w''$ denote the neighbours of $w$ in $P$, we have that each of $u$ and $v$ is adjacent to at most one (and therefore exactly one) of $w'$ and $w''$. By symmetry, we may assume that $u$ is adjacent to $w'$ and non-adjacent to $w''$, and that $v$ is adjacent to $w''$ and non-adjacent to $w'$. Now $(V(P) \cup \{u, v\}) \setminus \{w\}$ contains a subset of vertices inducing a copy of $H$; more precisely, $(V(P) \cup \{u, v\}) \setminus \{w\}$ is an induced subgraph of $G$ isomorphic to the graph obtained from $P$ by subdividing an edge in the component containing $w$.
\end{proof}

We also require an auxiliary graph construction, introduced as the the ``blob graph'' in \cite{gartland2021finding} and \cite{agrawal2024odd}.

\begin{definition}
    Let $G$ be a graph, and let $\mathcal{C}$ be a set of connected induced subgraphs of $G$. The \emph{blob graph} of $G$ and $\mathcal{C}$ is the graph $H = H(G, \mathcal{C})$ with $V(H) = \{v_C  :  C \in \mathcal{C}\}$ and for $v_C, v_{C'} \in V(H)$ we have $v_Cv_{C'} \in E(H)$ if and only if $V(C) \cap V(C') \neq \emptyset$ or there exist vertices $u \in V(C)$ and $u' \in V(C')$ such that $uu'\in E(G)$.
\end{definition}

Our use of $H(G, \mathcal{C})$ is analogous to that in \cite{agrawal2024odd}: In Lemma \ref{lem:p5freenice}, we show that if $G$ is $(P_5+rK_1)$-free, so too is $H$. Much later, in Lemma \ref{lem:givenCcanfindOPT}, we show that we can reduce solving WML$k$CIS on $G$ to solving \textsc{Maximum Weight Independent Set} on {$H(G,\mathcal{C})$ (with an appropriately chosen $\mathcal{C}$)}, which can be solved in polynomial time on graphs in $\mathcal{G}_r$ by Lemma \ref{lem:stable}. We first prove a result about $J$-free graphs for more general $J$, from which Lemma \ref{lem:p5freenice} follows as a corollary. 

Given a graph $G,$ we say vertices $u, v \in V(G)$ are \textit{true twins} if $N(u)\setminus \{v\} = N(v) \setminus \{u\}$ and $uv \in E(G)$. 

\begin{lemma}\label{lemma:ifJniceHisJfree}
Let $G$ and $J$ be graphs, let $\mathcal{C}$ be a set of connected induced subgraphs of $G$, and let $H = H(G,\mathcal{C})$. 
If all of the following properties hold, then $H$ is $J$-free. 
    \begin{enumerate}
        \item[(i)] $G$ is $J$-free, 
        \item[(ii)] $J$ has no true twins, and
        \item[(iii)] $J$-free graphs are closed under edge contraction.
    \end{enumerate}
\end{lemma}
\begin{proof}
    Suppose $G$ and $J$ are graphs, $\mathcal{C}$ is a set of connected induced subgraphs of $G$, and claims (i)--(iii) hold for $G$ and $J$. Let $H = H(G, \mathcal{C})$.  We construct an alternate graph $H''$ from $G$ satisfying the above properties, then show that $H''$ is isomorphic to $H$. 
    
    To that end, we first need to define an auxiliary graph $H'$ from which we will obtain $H''$ via edge contraction. Let $H'$ be the graph with vertex set $V(H') = \{(C, u)  :  C \in \mathcal{C}, u \in V(C)\}$ and edge set $E(H') = \{(C, u)(C',u') : u = u' \text{ or } uu' \in E(G) \}$. We call $u$ the \emph{base vertex} of $(C, u)$. Note that every two vertices in $H'$ with the same base vertex are true twins in $H'$. 

    Suppose for a contradiction that a vertex subset $W \subseteq V(H')$ induces a copy of $J$ in $H'$. Since $J$ has no true twins by (ii), it follows that $H'[W]$ has no true twins, and so all vertices in $W$ have unique base vertices in $G$, as established above. However, by definition, two vertices in $H'$ with different base vertices are adjacent if and only if their base vertices are adjacent in $G$, so the set $\{u \in V(G) : (C, u) \in V(W) \textnormal{ for some } C \in \mathcal{C}\}$ induces a copy of $J$ in $G$, contradicting (i). Hence $H'$ is $J$-free.  

    Finally, let $H''$ be the graph obtained from $H$ by contracting, for each $C \in \mathcal{C}$, all edges with both ends in the set $\{(C, u)  :  u \in V(C)\}$. Since each $C \in \mathcal{C}$ is connected, it follows that all vertices in $\{(C, u)  :  u \in V(C)\}$ are contracted to a single vertex $v_C$ in $H''$. We now show $H''$ is isomorphic to $H$, thus completing the proof. By (iii), since $H''$ is obtained from $H'$ via edge contractions and $H'$ is $J$-free, it follows that $H''$ is also $J$-free. Note that two contracted vertices $v_C$ and $v_C'$ in $H'$ are adjacent if and only if there exist adjacent vertices $(C, u)$ and $(C', u')$ in $H$, which occurs if and only if $u = u'$ or $uu' \in E(G)$. Additionally, $u = u'$ implies $V(C) \cap V(C') \not = \emptyset$. Therefore $H''$ is isomorphic to $H$: Both graphs have vertex sets in bijective correspondence with $\mathcal{C}$, and the edge relationships are identical. Therefore,  $H$ is also $J$-free, as desired.
\end{proof}

 Lemmas \ref{lemma:P5+rk1freegraphsnice} and \ref{lemma:ifJniceHisJfree} immediately imply the following result. 
\begin{lemma}
\label{lem:p5freenice}
    Given a graph $G \in \mathcal{G}_r$ and a set $\mathcal{C}$ of connected induced subgraphs of $G$, the graph $H(G,\mathcal{C})$ is also in $\mathcal{G}_r$.
\end{lemma}
\begin{proof}
    For all $r$, the graph $P_5+rK_1$ contains no true twins, and by Lemma \ref{lemma:P5+rk1freegraphsnice}, we have that $(P_5+rK_1)$-free graphs are closed under edge contraction. Hence it follows from Lemma \ref{lemma:ifJniceHisJfree} that $H(G,\mathcal{C})$ is $(P_5+rK_1)$-free. 
\end{proof}

Our solution to MW$k$CIS consists a number of (pairwise anticomplete) $L$-colourable connected components, each of which will be selected from a larger set $\mathcal{C}$ of induced subgraphs of our underlying graph, $G$. The auxiliary graph $H(G,\mathcal{C})$ is used to select, in polynomial time, a maximum-weight set of pairwise anticomplete components from $\mathcal{C}$. 

\begin{lemma}\label{lem:givenCcanfindOPT}
    Let $G \in \mathcal{G}_r$ and let $L$ be a list assignment such that $L(v) \subseteq [k]$ for all $v \in V(G)$. Let $\mathcal{C}$ be a list of (some of the) connected $L$-colourable induced subgraphs of $G$. Suppose that there exists a maximum-weight $L$-colourable induced subgraph of $G$, denoted by $OPT$, such that each component of $OPT$ is in $\mathcal{C}$. Then the WML$k$CIS problem can be solved for $G$ in polynomial time (in $|\mathcal{C}|$ and $v(G)$). 
\end{lemma}
\begin{proof}
    We construct the auxiliary graph $H = H(G,\mathcal{C})$ with running time polynomial in $|\mathcal{C}|$ and $v(G)$. By Lemma \ref{lem:p5freenice} it holds that $H \in \mathcal{G}_r$. For each $v_C \in V(H)$ we define $\omega(v_C) = \sum_{v \in V(C)} \omega(v)$. It suffices to find a maximum weight independent set (MWIS) in $H$ with respect to these weights. 

    Since $r$ is fixed, Lemma \ref{lem:stable} implies that we can find a maximum-weight independent set in $H$ in polynomial time (with respect to $v(H)$). As $v(H)=|\mathcal{C}|$, this can be done in polynomial time with respect to $|\mathcal{C}|$.
\end{proof}

The principal work of Section \ref{sec:p5rk1} is to produce the list $\mathcal{C}$ of ``candidate'' components. We end this section with two lemmas about important subgraphs of $(P_5+rK_1)$-free graphs. The first lemma concerns the identification of small dominating subgraphs in $k$-colourable induced subgraphs of $(P_5+rK_1)$-free graphs. For this, we will require the following result due to Camby and Schaudt \cite{camby2016new}.

\begin{lemma}[Camby \& Schaudt \cite{camby2016new}]\label{lem:cambyschaudt}
    If $C$ is a connected, $P_5$-free graph, then $C$ contains a connected, dominating subgraph $S$ that is either a clique or has at most three vertices.  
\end{lemma}

Using Lemma \ref{lem:cambyschaudt}, we prove the following result which will be crucial for bounding the number of connected subgraphs we need to consider in the algorithm behind Theorem \ref{thm:mainthmp5}. 

\begin{lemma}\label{lem:smallconnecdomset}
    If $C$ is a connected induced $k$-colourable subgraph of a graph $H \in \mathcal{G}_r$, then $C$ has a connected dominating subgraph on at most $\max\{k, 3, (k+1)(r-1) + 5\}$ vertices.
\end{lemma}

\begin{proof}
    If $C$ is $P_5$-free, then the result follows from Lemma \ref{lem:cambyschaudt}. In particular, since $C$ is $k$-colourable, it does not contain a clique of size $k+1$, so some connected dominating subgraph contains at most $\max\{k, 3\}$ vertices. 
    
    We may assume, then, that $C$ contains an induced subgraph $P$ isomorphic to $P_5$, and accordingly $r \ge 1$. Let $C'$ be the graph induced by the vertex set $V(C) \setminus N \langle P \rangle$. Since $C \subseteq H$ and $H \in \mathcal{G}_r$, it follows that the independence number of $C'$ is at most $r-1$, and hence that $C'$ has at most $r-1$ components. Since $C$ is by assumption $k$-colourable (that is, its vertices can be partitioned into at most $k$ independent sets), it follows that $v(C') \le k(r-1)$. For each of the components of $C'$, choose a vertex in $N(P) \cap V(C)$ adjacent in $H$ to that component; call this set of vertices $M$. 
    
    Let $S = V(P) \cup V(C') \cup M$. Note that $C[S]$ is connected via $M$, dominates $C$, and $|S| \le 5 + k(r-1) + (r-1) = (k+1)(r-1) + 5$, as desired.
\end{proof}


\section{WML$k$CIS in $(P_5 + rK_1)$-free graphs}\label{sec:p5rk1}

This section contains the proof of Theorem \ref{thm:mainthmp5}. Before we begin, we give a brief outline of the main ideas in the proof, which follows the same general structure as that of Agrawal et al. \cite{agrawal2024odd}. Throughout this section, fix $k$ and $r$ to be non-negative integers. Let $G$ be a graph in $\mathcal{G}_r$, let $\omega$ be a weight function with $\omega: V(G) \rightarrow \mathbb{R}_{+}$, and let $L$ be a list-$k$-assignment for $G$. First, we enumerate all possibilities for the polynomially-many small connected induced subgraphs promised by Lemma \ref{lem:smallconnecdomset}. We then use Algorithm \ref{refining_algorithm} to turn those subgraphs (and some ancillary sets) into a polynomially-sized set $\mathcal{C}$ of $L$-colourable connected induced subgraphs of $G$. We next show that there exists a solution $OPT$ to the WML$k$CIS problem for $G$ such that each component of $OPT$ is in $\mathcal{C}$ (Lemma \ref{lem:solwithallinC}), and finally that such a solution $OPT$ can be found in polynomial time (Lemma \ref{lem:givenCcanfindOPT}).

Each element of $\mathcal{C}$ is necessarily a connected, induced $L$-colourable subgraph of $G$. We obtain each graph $C \in \mathcal{C}$ from a small vertex subset $S$ inducing a connected dominating subgraph of $C$, together with an $L$-colouring of $G[S]$, and a number of ancillary sets. These form a \emph{canvas}: a tuple consisting of a subset $S$ of $V(G)$ with at most $\max\{k, 3, (k+1)(r-1) + 5\}$ vertices such that $G[S]$ is connected, an $L$-colouring  $f$ of $G[S]$, and several sets of independent sets of $V(G)$ that are used to extend $f$ to a subset of $N(S)$ and capture key properties of a possible WML$k$CIS solution. The definition is followed by a detailed explanation of each part. After defining canvases below, we show in Lemma \ref{lem:Csmall} that there are a polynomial number of canvases of $G$. We then produce a connected induced $L$-colourable subgraph $C$ of $G$ from each canvas (Algorithm \ref{refining_algorithm} and Lemma \ref{lem:solqvalid}). Ultimately, this graph $C$ will be a candidate for a connected component of a solution to the WML$k$CIS problem. From the set of all such candidate components, we compute an optimal solution to WML$k$CIS in polynomial time (Lemma \ref{lem:givenCcanfindOPT}).

\begin{definition}\label{def:canvas}
Given a graph $G \in \mathcal{G}_r$ and a list-$k$-assignment $L$ for $G$, a \emph{canvas} is a tuple  
$$Q = \left(S,f,\{A_c\}_{c \in [k]}, \{B_c\}_{c \in [k]}, \{Y_{i,c,\ell}\}_{\substack{i,\ell \in [|S|] \\  i < \ell \\ c \in [k]}}, \{Z_{i,c}\}_{\substack{i \in [|S|] \\ c \in [k]}}\right) $$ such that the following ten properties hold: 
\begin{enumerate}
    \item[(Set $S$)] $S$ is an ordered subset of $V(G)$ with $t = |S| \leq \max\{k, 3, (k+1)(r-1) + 5\}$ such that $G[S]$ is connected. Let $v_1, \ldots, v_t$ be the ordering of the elements of $S$, and for all $1 \le i \le t$, let $X_i= N_{G}(v_i) \setminus (S \cup (\bigcup_{1 \le j < i} X_j)).$ 

    \item[(Col.\ $f$)] The function $f$ is an $L$-colouring of  $G[S]$.

    \item[(Sets $A$)] For every colour $c \in [k]$, $A_c$ is an independent set of at most $2k$ vertices in $N(S)$ where $c \in L(v)$ for each $v \in A_c$. 
    
    \item[(Sets $B$)] For every colour $c \in [k]$, $B_c$ is an independent set of at most $r$ vertices in $V(G) \setminus N\langle S \rangle$ where $c \in L(v)$ for each $v \in B_c$.
    
    \item[(Sets $Y$)] For every pair of distinct indices $i, \ell \in [t]$ with $i < \ell$ and every colour $c \in [k]$, the set $Y_{i,c,\ell} \subseteq X_i$ is an independent set of at most two vertices in $N(S)$ where $c \in L(v)$ for each $v \in Y_{i,c,\ell}$. 
    
    \item[(Sets $Z$)] For every index $i \in [t]$ and every colour $c \in [k]$, $Z_{i,c} \subseteq X_i$ is an independent set of at most $r$ vertices where $c \in L(v)$ for each $v \in Z_{i,c}$. 

    \item[(Small $Z$)] If there exists some index $i \in [t]$ and colour $c \in [k]$ such that $|Z_{i, c}| < r$, then for all $\ell \in [t]$ with $\ell > i$, we have $Y_{i, c, \ell} \subseteq Z_{i, c}$, and $A_c \cap X_i \subseteq Z_{i, c}.$
    
    \item[(Disjoint)] For every pair of distinct colours $c, \overline{c} \in [k]$, we have that $A_{\overline{c}} \cap A_c = \emptyset, \; B_{\overline{c}} \cap B_c = \emptyset, \; Z_{i,\overline{c}} \cap Z_{i,c} = \emptyset,$ and $Y_{i,\overline{c},\ell} \cap Y_{i,c,\ell} = \emptyset$. Additionally, for every pair of indices $i, \ell \in[t]$ with $i < \ell$ and every pair of distinct colours $c, \overline{c} \in [k]$, we have that $A_c \cap Y_{i, \overline{c}, \ell} = \emptyset, \; A_c \cap Z_{i, \overline{c}}= \emptyset,$ and $Y_{i, c, \ell} \cap Z_{i, \overline{c}} = \emptyset$.
    
    \item[(Colour)]
    For every colour $c$, we have that $\{v_i \in S : f(v_i) = c\} \cup A_c \cup \bigcup_{\substack{i, \ell \in [t] \\ i < \ell}} Y_{i, c, \ell} \cup Z_{i, c}$  is an independent set. 
    

    \item[(Comps.)] The set $\bigcup_{c \in[k]}B_c$ is anticomplete to $\bigcup_{c \in [k]} \left(A_{c} \cup \bigcup_{\substack{i, \ell \in [t] \\ i < \ell}} Y_{i, c, \ell} \cup Z_{i, c}\right)$.
\end{enumerate}
\end{definition}

Here we will give some motivation for each condition in the definition of a canvas. Recall that from each canvas we produce a candidate connected component $C$ for a solution to the WML$k$CIS problem. By checking all possible canvases, we eventually find some canvas from which we can obtain an optimum solution. Each of the sets in Definition \ref{def:canvas} is a guess at some of the vertices and their colours in an $L$-colouring of an optimum solution. 

In (Set $S$), the ordered set $S$ is a guess at the vertex set of a small connected, dominating subgraph of $C$, guaranteed to exist by Lemma \ref{lem:smallconnecdomset}. For the correct choice of $S$, we have $V(C) \subseteq N\langle S \rangle$. The sets $X_i$ partition $N(S)\setminus V(S)$: Each vertex in $N(S)\setminus V(S)$ is placed into a set $X_i$ depending on its first neighbour in the ordering of $V(S)$. In Algorithm \ref{refining_algorithm} we find the remainder of component $C$ by considering each set $X_i$ individually. We also guess an $L$-colouring $f$ of $S$ in condition (Col. $f$). 

For each colour $c$, the set $A_c$ is a guess of at most $2k$ vertices coloured $c$ in the subgraph $C$ that are not contained in $S$. As such, condition (Sets $A$) specifies that each $A_c$ is an independent set and each vertex in $A_c$ can be coloured $c$. Likewise in condition (Sets $B$), the set $B_c$ is our guess of at most $r$ vertices coloured $c$ in \emph{other} components of the solution. If $|B_c| < r$, then we assume that $B_c$ includes all vertices coloured $c$ in other components of the solution. Condition (Comps.) ensures that these sets $B_c$ are anticomplete to all vertices of our guess in $C$, since distinct components anticomplete. The sets $A_c$ and $B_c$ encode all the information that we need to ensure that componenents computed from this canvas are ``compatible'' with the remainder of the solution; this will be made precise in Lemma \ref{lem:solwithallinC}. 

For every pair of distinct indices $i, \ell \in [t]$ with $i < \ell$ and every colour $c$, the set $Y_{i, c, \ell}$ referenced in condition (Sets $Y$) is a set of at most two vertices in $C$ coloured $c$ that are contained in $X_i$. As detailed in condition (Sets $Z$), the set $Z_{i, c}$ is our guess of a set of at most $r$ vertices from $X_i$ that are coloured $c$ in $C$. The purpose of the sets $Y_{i, c, \ell}$ and $Z_{\ell, c}$ is to provide enough information about the vertices in $X_i \cap C$ and $X_\ell \cap C$ which receive colour $c$ to be able to independently pick the remaining vertices of colour $c$ in $X_i \cap C$ and $X_\ell \cap C$. If there are fewer than $r$ vertices coloured $c$ in $X_i$, we will assume later on that $Z_{i, c}$ is the set of \emph{all} vertices in $C$ coloured $c$ that are contained in $X_i$.  Condition (Small $Z$) ensures that the rest of our guess is consistent with this assumption. 

To ensure our guesses ultimately result in a proper $L$-colouring of $C$, condition (Disjoint) guarantees that we guess at most one colour per vertex and condition (Colour) guarantees that each of our guessed colour classes is an independent set. 

Recall that $k$ and $r$ are fixed constants (and treated as such when computing running times and sizes of sets). In what follows, let $G$ be a fixed graph in $\mathcal{G}_r$ and let $L$ be a fixed list-$k$-assignment for $G$. Let $\mathfrak{Q}_L$ be the set of canvases for $G$ formed using the list assignment $L$. In the following lemma we show that $\mathfrak{Q}_L$ is polynomially-sized. 

\begin{lemma}\label{lem:Csmall}
    $|\mathfrak{Q}_L| \in \textnormal{poly}(v(G))$, and $\mathfrak{Q}_L$ can be computed in \textnormal{poly}$(v(G))$-time.
\end{lemma}
\begin{proof}
There are polynomially many subsets of $V(G)$ of order at most ${\max\{k, 3, (k+1)(r-1) + 5\}}$. Each of these sets has a constant number of orderings, so there are polynomially many ordered sets $S$ from which canvases can be built. For each ordered set $S$, since $k$ is a constant, the number of $L$-colourings of $G[S]$ is also bounded by a constant. We can enumerate all tuples satisfying the first six conditions in polynomial time, and then check in polynomial time which of them also satisfy the last four conditions. Hence for each ordered set $S$ there are polynomially many canvases $Q$ whose first entry is $S$, and they can be computed in polynomial time.
\end{proof}

Below, we define a number of auxiliary sets related to a canvas $Q$. In Definition \ref{def:associated}, we use these auxiliary sets to define a relation between graphs and canvases which we call \emph{being associated with} a canvas. This relation formalizes the way our guessed component $C$ intersects with the elements of the canvas, as outlined after Definition \ref{def:canvas}.

\begin{definition}\label{defs:QSoperators}
    Consider a canvas $Q = \left(S,f,\{A_c\}_{c \in [k]}, \{B_c\}_{c \in [k]}, \{Y_{i,c,\ell}\}_{\substack{i,\ell \in [|S|] \\  i < \ell \\ c \in [k]}}, \{Z_{i,c}\}_{\substack{i \in [|S|] \\ c \in [k]}}\right) $. We write $v_1, \dots, v_{t}$ for the vertices of $S$ in order; and for each $i \in [t]$, we let $X_i$ be the set $N_{G}(v_i) \setminus \left(S \cup \bigcup_{1 \le j < i} X_j\right)$. Let $M := V(G) \setminus N\langle S \rangle$ and let $W := \{v \in M :  N_{G}(v) \cap B_c \neq \emptyset \text{ for every } c\in [k] \text{ such that } |B_c| = r\} \setminus \bigcup_{c \in [k]} B_c$. 
\end{definition}

The motivation behind the definitions of the sets $M$ and $W$ is as follows. Suppose the canvas $Q$ produces a component $C$ of an optimum solution $OPT$. The vertices in other components of $OPT$ have no neighbour in $C$. Since $S$ is a dominating subset of $V(C)$, it follows that vertices of $V(OPT) \setminus V(C)$ are contained in $M$. Furthermore, recall that $B_c$ is a set of vertices coloured $c$ in $OPT\setminus C$. Thus, if a vertex $v$ has a neighbour in $B_c$ for some $c$, then if $v$ is in $OPT$, 
it follows that $v$ is not coloured $c$. By assumption, if $|B_c| < r$, then $B_c$ includes all vertices in $OPT$ coloured $c$. Thus, if $v$ has a neighbour in all sets $B_c$ of order $r$ and is not contained in any of the sets $B_c$, then $v$ is not in OPT. That is: $W$ is a subset of the vertices in $M$ that are not contained in $OPT$.

\begin{definition}
\label{def:associated}
    Given a canvas $ Q =\left(S,f,\{A_c\}_{c \in [k]}, \{B_c\}_{c \in [k]}, \{Y_{i,c,\ell}\}_{\substack{i,\ell \in [|S|] \\  i < \ell \\ c \in [k]}}, \{Z_{i,c}\}_{\substack{i \in [|S|] \\ c \in [k]}}\right) $ of a graph $G$ with list assignment $L$, we say that an $L$-colourable induced subgraph $C \subseteq G$ is \emph{associated with $Q$} if the following conditions all hold:
    \begin{enumerate}[(i)]
        \item $S \subseteq V(C)$; 
        \item $S$ dominates $C$; 
        \item there exists an $L$-colouring $f'$ of $C$ such that $f'|_S = f$; and 
        \item for every $i \in [t]$ and $v \in V(C) \cap X_i$, the following conditions hold: 
        \begin{enumerate}[(a)]
            \item $v$ is anticomplete to $\bigcup_{c \in [k]} B_c$;
            \item for each $c \in [k]$ and each $j,\ell \in [t]$ with $i < \ell$, if $v \in A_c \cup Y_{i,c,\ell} \cup Z_{j,c}$ we have that $f'(v) = c$; and 
            \item if $v$ is not included in any $A_c$, $Y_{i,c,\ell}$, or $Z_{j,c}$, then $f'(v)$ is not an element of any of the following sets: 
    \begin{enumerate}[1.]
        \item $\{f(v_j) : j \in[t] \text{ and } vv_j \in E(G) \}$
        \item $\{c: (N(v) \cap M) \setminus W \not \subseteq N(A_c)\}$ 
        \item $\left\{c: v \in N\left(A_c \cup B_c\cup \bigcup_{\ell \in [t]}Z_{\ell, c} \cup \bigcup_{\substack{\ell, j \in [t], \ell < j}}Y_{\ell, c, j}\right)\right\}$ \item $\{c: |Z_{i, c}| < r\}$
        \item $\{c: \exists j > i \text{ such that } {|Z_{j,c}| = r \text{ and } }(N(v) \cap X_j)\setminus N(Z_{j, c})\not\subseteq (N(Y_{i, c, j})\cap X_j) \setminus N(Z_{j, c})\}$.
    \end{enumerate} 
        \end{enumerate} 
    \end{enumerate}
\end{definition}

These conditions are used to forbid colours at $v$, justified as follows. 
\begin{enumerate}[1.]
    \item We assume that $S$ is coloured by $f$ in our output component, so we forbid $f(v_j)$ for every neighbour $v_j$ of $v$ in $S$.
    \item Second, our plan is to arrange that $N(A_c)$ contains all vertices in $M \setminus W$ that have a neighbour of colour $c$ in $C$ (in other words, $N(v) \cap (M \setminus W) \subseteq N(A_c)$ for all $v \in C$ with $f'(v) = c$). Therefore, vertices in $C$ with a neighbour in $(M \setminus W) \setminus N(A_c)$ will not be coloured $c$. 
    \item For each colour $c \in [k]$, we assume vertices in the sets $A_c, B_c$, $Z_{\ell,c}$ and $Y_{\ell, c, j}$ for $\ell, j \in [t]$ with $\ell < j$ are coloured $c$ in our $k$-colouring of $C$, and so colour $c$ is forbidden from neighbours of these sets. 
    \item For each $c \in [k]$ and $i \in [t]$ with $|Z_{i, c}| < r$, no vertices in  $X_i \setminus Z_{i,c}$ are coloured $c$ (condition (Small $Z$) of Definition \ref{def:canvas}). Since we assumed $v$ is not in $Z_{i,c}$, we forbid colour $c$ for $v$. 
    \item Our plan is to arrange that for all $i < j$ and all $c \in [k]$ such that $|Z_{j,c}| = r$, the set $N(Y_{i, c, j})$ contains all vertices in $X_j \setminus N(Z_{j,c})$ with a neighbour of colour $c$ in $X_i \cap V(C)$ (in other words, $N(v) \cap (X_j \setminus N(Z_{j,c})) \subseteq N(Y_{i, c, j})$ for all $v \in V(C) \cap X_i$ such that $f'(v) = c$). Therefore, if a vertex $v$ in $X_i \cap C$ has a neighbour in $(X_j \setminus N(Z_{j,c})) \setminus N(Y_{i,c,j})$, then $v$ should not receive colour $c$. 
\end{enumerate}

In Algorithm \ref{refining_algorithm}, we follow the strictures of Definition \ref{def:associated} to obtain a maximum-weight associated graph $C$ from $Q$. We prove in Lemma \ref{lemma:alg_terminates} that Algorithm \ref{refining_algorithm} terminates in polynomial time, predicated on the inductive hypothesis that WM$(k-1)$CIS is solvable in polynomial time. The correctness of Algorithm \ref{refining_algorithm} is proven in Lemmas \ref{lem:solqvalid} and \ref{lem:c'asgoodasc}. Given a graph $H \in \mathcal{G}_r$, a list assignment $L$ for $H$ with $\left|\bigcup_{v \in V(H)} L(v) \right| \leq k-1$, and a weight function $\omega: V(H) \rightarrow \mathbb{R}_{+}$, let $U(H,L, \omega)$ be an optimum solution to the WM$(k-1)$CIS problem on $H$. 

\begin{algorithm}
\caption{Component Generation Algorithm}\label{refining_algorithm} 
\begin{algorithmic} [1]
\Require $H \in \mathcal{G}_r$ and $Q = (S,f,\{A_c\}_{c \in [k]}, \{B_c\}_{c \in [k]}, \{Y_{i,c,\ell}\}_{\substack{i,\ell \in [|S|],   i < \ell, c \in [k]}}, \{Z_{i,c}\}_{\substack{i \in [|S|], c \in [k]}}) $ a canvas of $H$ with $S = \{v_1, \dots, v_t\}$ for some $t \geq 1$, a list assignment $L$ with $\bigcup_{v \in V(H)} L(v) = [k]$, and a weight function $\omega: V(H) \rightarrow \mathbb{R}_{+}$. 
\Ensure A connected induced $k$-colourable subgraph $C$ of $H$.
\ForAll{$v \in N(S)$} \label{line:2}
\State $i:= p \in [t]$ such that $  v \in X_p$ \label{line:3}
\If{there exists $c \in [k]$ such that $v \in A_c \cup \left(\bigcup_{j \in [t], j > i}Y_{i, c, j} \right)\cup \left(\bigcup_{i \in [t]} Z_{i, c}\right),$} \label{line:4} 
    \State $F(v) := \{\overline{c} \in [k] :  \overline{c} \neq c\}$ \label{line:5}
    \Else
    \State
        $F(v) :=  \{f(v_j) : j \in[t] \text{ and } vv_j \in E(H) \} \text{ }\cup $ \label{line:7}\\ $ \hspace{25.5mm} \{c: N(v) \cap (M \setminus W) \not \subseteq N(A_c)\} \text{ }\cup $\label{line:8} \\ $\hspace{25.5mm} \{c: v \in N\left(A_c \cup B_c \cup \bigcup_{\ell \in [t]}Z_{\ell, c} \cup \bigcup_{\substack{\ell, j \in [t], \ell < j}}Y_{\ell, c, j}\right)\} \text{ }\cup $ \label{line:9}\\ $\hspace{25.5mm} \{c: |Z_{i, c}| < r\}\text{ }\cup $ \label{line:10}\\ $\hspace{25.5mm} \{c: \exists j > i \text{ s.t.\ } |Z_{j,c}| = r \text{ and } (N(v) \cap X_j)\setminus N(Z_{j,c})\not\subseteq (N(Y_{i, c, j})\cap X_j) \setminus N(Z_{j, c})\}$\label{line:11} 
\EndIf\label{line:12}
\State $L(v) \leftarrow L(v) \setminus F(v)$ \label{line:13}
\EndFor

\State \Return $H\left[S \cup \bigcup_{i \in [t]} V\left(U\left(X_i \setminus {N\left(\bigcup_{c \in [k]} B_c\right)}, L, \omega \right)\right)\right]$ \label{line:15}

\end{algorithmic}
\end{algorithm}

We outline the algorithm briefly, and then provide a line-by-line explanation. In the body of the algorithm, we modify the list assignment $L$ by removing colours from a vertex's list that would contradict our known restrictions about the structure of the output component. The output of the algorithm in line 14 is the graph induced by $S$ together with a solution to the WM$(k-1)$CIS problem on (a subset of) each $X_i$ with the given restricted list colouring. Note this latter solution is found by induction, using that $f(v_i)$ is removed from the lists of all vertices in $X_i$.

We now give a line-by-line explanation of the algorithm. Given a graph $H$, canvas $Q$, and list-assignment $L$, we consider vertices and colours that could possibly be in an output component $C$ associated with $Q$. From Definition \ref{def:associated}, we know that $S \subseteq V(C)$ and $V(C) \setminus S \subseteq N(S)$, so we only consider vertices in $N(S)$ (line \ref{line:2}). 

For each vertex $v$, we identify the unique index $i$ such that $v \in X_i$ (line \ref{line:3}). In lines \ref{line:4}-\ref{line:12}, we create sets of forbidden colours for $v$, which allows us to verify, among other properties, that this union does indeed produce an $L$-colourable induced subgraph of $G$. 

If there is a colour $c$ such that $v$ is contained in $A_c \cup \left(\bigcup_{i, j \in [t], i< j}Y_{i, c, j}\right) \cup \left(\bigcup_{i \in [t]}Z_{i, c}\right)$, then we assume $v$ is coloured $c$ in our output component and forbid all other colours. (Note that by condition (Disjoint) in Definition \ref{def:canvas}, $v$ is in this union of sets for at most one $c\in [k]$.) Otherwise, we forbid several different sets of colours for $v$, in accordance with Definition \ref{def:associated} (lines \ref{line:7}-\ref{line:11}).

We update the list assignment to take into account which colours are forbidden for each vertex (line \ref{line:13}). Note that by line \ref{line:7}, in each set $X_i$, there are at most $k-1$ colours in $\bigcup_{v \in X_i}L(v)$ (with $L$ as updated in line 12). Note that we can relabel colours so that $\bigcup_{v \in X_i}L(v) \subseteq [k-1]$, so an optimum solution to the WM$(k-1)$CIS problem on an induced subgraph of $H[X_i]$ can be found in polynomial time by the inductive hypothesis.  {In line \ref{line:15}, we remove neighbours of sets $B_c$ from $X_i$; since vertices in $B_c$ are in another component of OPT, they are not adjacent to vertices in $C$.} Our solution is built from the subgraph induced by the set $S$ together with the union of the vertex sets of the solutions for these subsets of $X_i$. 

The following three lemmas establish that Algorithm \ref{refining_algorithm} terminates in polynomial time and produces a maximum-weight connected induced $k$-colourable subgraph associated with $Q$.

\begin{lemma}
\label{lemma:alg_terminates}
    If WM$(k-1)$CIS (with input list assignment $L$ such that $L(v) \subseteq [k-1]$ for every vertex $v$) is solvable in polynomial time, then Algorithm \ref{refining_algorithm} terminates in polynomial time with inputs $H$ and $Q$.
\end{lemma}
\begin{proof}
    The computations in lines \ref{line:2}-\ref{line:13} are immediately polynomial in terms of $v(H)$, since each line is in $O(v(H))$. We assume that WM$(k-1)$CIS is solvable in $\textnormal{poly}(v(H))$-time. Thus the graphs $U\left(X_i \setminus {N\left(\bigcup_{c \in [k]} B_c\right)}, L,\omega \right)$ in line \ref{line:15} are generated in $\textnormal{poly}(v(H))$-time, so Algorithm \ref{refining_algorithm} terminates in $\textnormal{poly}(v(H))$-time. 
\end{proof}

We show now that the graphs generated by Algorithm \ref{refining_algorithm} are connected induced $L$-colourable subgraphs.

\begin{lemma}\label{lem:solqvalid}
     If $C$ is generated from Algorithm \ref{refining_algorithm} with inputs $H, Q,$ and $L'$, then $C$ is a connected induced $L'$-colourable induced subgraph of $H$.
\end{lemma}
\begin{proof}
    Let $C$ be generated from Algorithm \ref{refining_algorithm} with inputs $H, Q,$ and $L'$. 
    Recall that for $i \in [t]$, the graph $U\left(X_i \setminus {N\left(\bigcup_{c \in [k]} B_c\right)}, L,\omega \right)$ is a solution to the WM$(k-1)$CIS problem on the graph $H\left[X_i {\setminus N\left(\bigcup_{c \in [k]} B_c\right)}\right]$ with respect to the list assignment $L$ defined in line \ref{line:13} of the algorithm. 
    Note that $$C = H\left[S \cup \bigcup_{i \in [t]} V\left(U\left(X_i{\setminus N\left(\bigcup_{c \in [k]} B_c\right)},L,\omega\right)\right)\right],$$ so $C$ is an induced subgraph of $H$. Additionally, since the sets $X_i$ partition $N(S)$, and since $S$ is connected by the definition of a canvas, we know the graph $C$ is connected. Since $L(v) \subseteq L'(v)$ for each vertex $v$, we also have that each graph $U\left(X_i{\setminus N\left(\bigcup_{c \in [k]} B_c\right)}, L,\omega\right)$ is $L'$-colourable by definition. Moreover, we claim that adjacent vertices in distinct subgraphs $U\left(X_i{\setminus N\left(\bigcup_{c \in [k]} B_c\right)}, L,\omega\right)$ and $U\left(X_j{\setminus N\left(\bigcup_{c \in [k]} B_c\right)}, L,\omega\right)$ always receive different colours. To see this, let $uv \in E(H)$ such that $u \in V\left(U\left(X_i{\setminus N\left(\bigcup_{c \in [k]} B_c\right)}, L,\omega\right)\right)$ and $v \in V\left(U\left(X_j{\setminus N\left(\bigcup_{c \in [k]} B_c\right)}, L,\omega\right)\right)$. Let $L(u)$ and $L(v)$ be the lists of $u$ and $v$ after the list assignment $L$ is updated in line \ref{line:13} of Algorithm \ref{refining_algorithm}. We show that $L(u)$ and $L(v)$ are disjoint. To that end, let $c \in L(u)$ {and suppose without loss of generality that $i < j$.} This implies either $|Z_{j,c}| < r$ or $(N(u) \cap X_j) \setminus N(Z_{j, c}) \subseteq N(Y_{i, c, j}) \setminus N(Z_{j, c})$ by line \ref{line:11}. If $|Z_{j,c}| < r$, then by line \ref{line:10}, $c \notin L(v)$. Otherwise, since $v \in N(u) \cap X_j$, we have that either $v \in N(Z_{j,c})$ or $v \in N(Y_{i, c, j})$, and so colour $c \notin L(v)$ by line \ref{line:9}. The lists are disjoint, as desired, and so the $L'$-colouring is proper.
\end{proof}

Finally, we show that the algorithm outputs a subgraph of maximum weight with respect to all possible subgraphs associated with a canvas $Q$. 

\begin{lemma}\label{lem:c'asgoodasc}
    If $C$ is generated from Algorithm \ref{refining_algorithm} with inputs $H, Q$ and $L'$, then $C$ is a maximum-weight induced subgraph associated with $Q$. 
\end{lemma}
\begin{proof}
Let $C$ be generated from Algorithm \ref{refining_algorithm} with inputs $H, Q$ and $L'$. Let $L$ be the list assignment produced in line \ref{line:13}, and let $f'$ be a colouring of $C$ with the properties that $f'$ and $f$ agree on $S$ and $f'(v) \in L(v)$ for all $v \in V(C)$. By construction, this colouring exists, $S$ dominates $C$ and $S \subseteq V(C)$, and $f'$ meets the colouring restrictions of Definition \ref{def:associated}. Hence $C$ is associated with $Q$. Let $C'$ be another subgraph associated with $Q$, with $L'$-colouring $f''$, say. We can write $V(C') = S \cup \bigcup_{i \in [t]} \left(V(C') \cap \left(X_i{\setminus N\left(\bigcup_{c \in [k]} B_c\right)}\right)\right)$, since the vertices of any graph associated with $Q$ are drawn from $N\langle S \rangle$ and avoid the neighbour sets of $B_c$ for all $c \in [k]$. The list assignment $L$ generated by Algorithm \ref{refining_algorithm} matches the restrictions of Definition \ref{def:associated}, and hence $f'$ is an $L$-colouring of $C'$. Since $U\left(X_i{\setminus N\left(\bigcup_{c \in [k]} B_c\right)},L,\omega\right)$ is a maximum-weight $L$-colourable induced subgraph of $X_i$, it follows that $\omega\left(U\left(X_i{\setminus N\left(\bigcup_{c \in [k]} B_c\right)},L,\omega\right)\right) \ge \omega(C'[X_i])$. Hence $$\omega(C) = \omega(S) + \sum_{i \in [t]} \omega\left(U\left(X_i{\setminus N\left(\bigcup_{c \in [k]} B_c\right)},L,\omega\right)\right) \ge \omega(S) + \sum_{i \in [t]} \omega(C'[X_i]) = \omega(C'),$$ so we conclude that $C$ is maximum.
\end{proof}

Fix a list-$k$-assignment $L$ of $G$. Let $\mathcal{C}$ be the set of outputs of Algorithm \ref{refining_algorithm} evaluated on $G$ and each canvas $Q \in \mathfrak{Q}_L$. By Lemmas \ref{lem:solqvalid} and \ref{lem:c'asgoodasc}, each graph $C \in \mathcal{C}$ is \cece{a} maximum-weight connected induced $L$-colourable subgraphs associated to its input canvas. We now show that {we} can construct a solution to the WML$k$CIS problem using only the components of $\mathcal{C}$. 

\begin{lemma}\label{lem:solwithallinC}
    There exists an optimum solution $OPT$ for the WML$k$CIS problem with input graph $G$ and list assignment $L$ such that every component $C$ of $OPT$ is in $\mathcal{C}$. 
\end{lemma}
\begin{proof}
   We prove something slightly stronger: that there exists an optimum solution $OPT$ such that all its components are in $\mathcal{C}$, and, furthermore, each component satisfies two additional properties. Given a component $C$ of $OPT$ with respect to its input canvas $Q$, for $c \in [k]$, let $A'_c$ be the set of vertices in $C \setminus S$ coloured $c$. Likewise, for $i,\ell \in [t]$ with $i < \ell$ and $c \in [k]$, let $Y'_{i,c}$ be the set of vertices in $C \cap X_i$ coloured $c$. We will show that component $C$ satisfies
    $$N(A_c) \cap (M \setminus W) = N(A'_c) \cap (M \setminus W)$$ 
    for all $c \in [k]$, and 
    $$(N(Y_{i,c,\ell}) \cap X_\ell) \setminus N(Z_{\ell,c}) = (N(Y'_{i,c}) \cap X_\ell) \setminus N(Z_{\ell,c})$$ 
     for all $i,\ell \in [t]$ with $i < \ell$ and all $c \in [k]$.
    
    Let us say that $C$ is \emph{excellent} if $C \in \mathcal{C}$ and $C$ satisfies the two properties above. As described after Definition \ref{def:associated}, these conditions allow us to forbid colours for certain vertices in the algorithm (lines \ref{line:8}, \ref{line:11}), which in turn aids in our proofs of Lemma \ref{lem:solqvalid} and this lemma. Suppose for a contradiction that no optimum solution exists where all components are excellent. Instead, let $OPT$ be a solution to the WML$k$CIS problem for $G$ that is maximum with respect to the number of excellent components. Let $C$ be a component of $OPT$ that is not excellent. We will produce an excellent replacement component $C'$ for $C$ such that $(OPT \setminus C) \cup C'$ is an $L$-colourable induced subgraph of weight at least as high as that of $OPT$. 

    Let $f$ be an $L$-colouring of $OPT$. We define a canvas $$Q = \left(S,f|_{G[S]},\{A_c\}_{c \in [k]}, \{B_c\}_{c \in [k]}, \{Y_{i,c,\ell}\}_{\substack{i,\ell \in [|S|] \\  i < \ell \\ c \in [k]}}, \{Z_{i,c}\}_{\substack{i \in [|S|] \\ c \in [k]}}\right) $$ for $C$ as follows:
    \begin{enumerate}
    \item [($S$)] By Lemma \ref{lem:smallconnecdomset}, the component $C$ contains a connected dominating subgraph with at most  ${\max\{k, 3, (k+1)(r-1) + 5\}}$ vertices; let $S =\{v_1, v_2, \dots, v_t\}$ be the ordered vertex set of such a dominating subgraph (where the ordering is chosen arbitrarily). For all $1 \le i \le t$, let $X_i= N_{G}(v_i) \setminus \left(S \cup \bigcup_{1 \le j < i} X_j\right).$ 

    \item [($B$)] For each $c \in [k]$, let $B_c$ be a set of vertices coloured $c$ in $M = V(G) \setminus N\langle S \rangle$, chosen as follows. If $M$ contains at least $r$ vertices coloured $c$, then let $B_c$ be any set of size $r$ of these vertices. Otherwise, let $B_c$ be the set of all vertices in $M$ coloured $c$. 
    
    \item [($A$)] Let $W$ be as in Definition \ref{defs:QSoperators}, that is,  $W := \{v \in M :  N_{G}(v) \cap B_c \neq \emptyset \text{ for every } c\in [k] \text{ such that } |B_c| = r\} \setminus \bigcup_{c \in [k]} B_c$.  For each colour $c \in [k]$, let $A'_c$ be the set of vertices $v$ in $N(S) \cap V(C)$ that are coloured $c$. Let $A_c$ be a minimal subset of $A'_c$ with the property that $N(A_c) \cap (M \setminus W) = N(A'_c) \cap (M \setminus W)$. If for all $c' \in [k]$ we have $|B_{c'}| < r$, we have $M \setminus W = \emptyset$, and so $A_c = \emptyset$. 

    \item [($Z$)] For each $i \in [t]$, let $Z_{i,c}$ be a set of vertices coloured $c$ in $X_i$, chosen as follows. If $X_i$ contains at least $r$ vertices coloured $c$, then let $Z_{i,c}$ be any set of size $r$ of these vertices. Otherwise, let $Z_{i,c}$ be the set of all vertices in $X_i$ coloured $c$.

    \item [($Y$)] For each triple $i,c,\ell$ with $i, \ell \in [t]$, {$i < \ell$}, and $c \in [k]$, let $Y'_{i,c}$ be the set of vertices in $X_i$ that are coloured $c$. If $|Z_{\ell,c}| = r$, then let $Y_{i,c,\ell}$ be a minimal subset of $Y'_{i,c}$ with the property that $(N(Y_{i,c,\ell}) \cap X_\ell) \setminus N(Z_{\ell,c}) = (N(Y'_{i,c}) \cap X_\ell) \setminus N(Z_{\ell,c})$. If $|Z_{\ell,c}| < r$,  then let $Y_{i,c,\ell} = \emptyset$.  
    \end{enumerate}

    Since $f$ colours only the vertices of $OPT$, it follows that $A_c, Z_{i,c}, $ and $Y_{i,c,\ell}$ are contained in $C$ for all index choices, and $B_c$ is contained in $OPT \setminus C$.
    
    Next, we show that $Q$ is a canvas. For the most part, this follows directly from Definition \ref{def:canvas}.    It remains to show that for all $i,c,\ell$ with $i < \ell$, we have that $|A_c| \leq 2k$ and $|Y_{i,c,\ell}| \leq 2$. To that end, we prove the following three claims. The situation described in the proof of Claim \ref{claim:aissmall} is illustrated in Figure \ref{fig:acsmall}.

\begin{figure}[ht]
\tikzset{black/.style={shape=circle,draw=black,fill=black,inner sep=1pt, minimum size=4pt}}
\tikzset{white/.style={shape=circle,draw=black,fill=white,inner sep=1pt, minimum size=16pt}}
\tikzset{invisible/.style={shape=circle,draw=black,fill=black,inner sep=0pt, minimum size=0.1pt}}
\tikzset{decoration={snake,amplitude=.7mm,segment length=2mm,
                       post length=0mm,pre length=0mm}}
\begin{center}
\
\begin{tikzpicture}[scale=0.7]

\filldraw[color=black!100, fill=black!0, very thick] (-1,0) ellipse (1 and 3); 
\node[] at (-1.4,3.3) {$S$};
\node[] at (-1.4,0) {$P$};

\filldraw[color=black!100, fill=black!0, very thick] (1.5,0) ellipse (1 and 3); 
\node[] at (1.7,3.3) {$A'_c$};

\filldraw[color=black!100, fill=black!0, very thick] (1.5,0) ellipse (0.8 and 2); 
\node[] at (1.2,2.3) {$A_c$};

\filldraw[color=black!100, fill=black!0, very thick] (4,0) ellipse (1 and 2.4); 
\node[] at (4.1,2.8) {$X $};

\filldraw[color=black!100, fill=black!0, line width=2.5] (7,0) ellipse (1 and 2); 
\node[] at (7.4,2.2) {$B_j$};

        \node[white] (s1) at (-1,1) {$s_1$};
        \node[white] (s2) at (-1,-1) {$s_2$};
        \node[black] (u1) at (1.5,1) {};
        \node[] at (1.5,1.4) {$u_1$};
        \node[black] (u2) at (1.5,-1) {};
        \node[] at (1.5,-1.4) {$u_2$};
        \node[black] (mu2) at (4,-1) {};
        \node[] at (4,-1.4) {$m(u_2)$};
        \node[black] (mu1) at (4,1) {};
        \node[] at (4,1.4) {$m(u_1)$};
        
        \node[invisible] (bjt) at (7,2) {};
        \node[invisible] (bjb) at (7,-2) {};

        \draw[dotted] (u1)--(mu2); 
        \draw[dotted] (u2)--(mu1);
        \draw[dotted] (mu2)--(mu1);
        \draw[dotted] (mu1)--(bjt); 
        \draw[dotted] (mu1)--(bjb); 
        \draw[dotted] (mu2)--(bjt); 
        \draw[dotted] (mu2)--(bjb); 
    
        \draw[line width=2.5] (s1) to (u1);
        \draw[line width=2.5] (s2) to (u2);
        \draw[line width=2.5] (u2) to (mu2);
        \draw[line width=2.5] (u1) to (mu1);

        \draw[decorate, line width=2] (s1)--(s2);
\end{tikzpicture}
\hskip 11mm
\begin{tikzpicture}[scale=0.7]

\filldraw[color=black!100, fill=black!0, very thick] (-1,0) ellipse (1 and 3); 
\node[] at (-1.4,3.3) {$S$};
\node[] at (-1.4,0) {$P$};

\filldraw[color=black!100, fill=black!0, very thick] (1.5,0) ellipse (1 and 3); 
\node[] at (1.7,3.3) {$A'_c$};

\filldraw[color=black!100, fill=black!0, very thick] (1.5,0) ellipse (0.8 and 2); 
\node[] at (1.2,2.3) {$A_c$};

\filldraw[color=black!100, fill=black!0, very thick] (4,0) ellipse (1 and 2.4); 
\node[] at (4.1,2.8) {$X $};

\filldraw[color=black!100, fill=black!0, line width=2.5] (7,0) ellipse (1 and 2); 
\node[] at (7.4,2.2) {$B_j$};

        \node[white] (s1) at (-1,1) {$s_1$};
        \node[white] (s2) at (-1,-1) {$s_2$};
        \node[black] (u1) at (1.5,1) {};
        \node[] at (1.5,1.4) {$u_1$};
        \node[black] (u2) at (1.5,-1) {};
        \node[] at (1.5,-1.4) {$u_2$};
        \node[black] (mu2) at (4,-1) {};
        \node[] at (4,-1.4) {$m(u_2)$};
        \node[black] (mu1) at (4,1) {};
        \node[] at (4,1.4) {$m(u_1)$};
        \node[black] (mu3) at (4,-.3) {};
        \node[] at (4,0.2) {$m(u_3)$};
         
        \node[invisible] (bjt) at (7,2) {};
        \node[invisible] (bjb) at (7,-2) {};

        \draw[dotted] (u1)--(mu2); 
        \draw[dotted] (u2)--(mu1);
        \draw[dotted] (u1)--(mu3);
        \draw[dotted] (u2)--(mu3);
        \draw[dotted] (mu1)--(bjt); 
        \draw[dotted] (mu1)--(bjb); 
        \draw[dotted] (mu2)--(bjt); 
        \draw[dotted] (mu2)--(bjb); 
        \draw[dotted] (mu3)--(bjb); 
        \draw[dotted] (mu3)--(bjt); 

        \draw[line width=2.5] (s1) to (u1);
        \draw[line width=2.5] (s2) to (u2);
        \draw[line width=2.5] (u2) to (mu2);
        \draw[black] (u1)--(mu1);
        \draw[line width=2.5] (mu3) to (mu2);

        \draw[decorate, line width=2] (s1)--(s2);
\end{tikzpicture}
\caption{The structure described in Claim \ref{claim:aissmall}. Here,   $X = N(A'_c) \cap (M\setminus W).$ Dashed lines indicate non-adjacency. The path $u_1Pu_2$ is the shortest $(u_1,u_2)$-path with internal vertices in $S$, and hence if $s_1 \neq s_2$, we have that $u_2s_1$ and $u_1s_2$ are not edges in $G$. To keep the image uncluttered, the dashed lines between $S$ and each of $X$ and $B_j$ as well as those between $A'_c$ and $B_j$ have been omitted. On the left, $m(u_1),u_1,P,u_2,m(u_2)$ together with $B_j$ contains an induced copy of $(P_5 + rK_1)$. On the right, $u_1,P,u_2,m(u_2),m(u_3)$ with $B_j$ contains an induced copy of $(P_5 + rK_1)$.  As both cases contradict that $G \in \mathcal{G}_r$, we conclude that $|A_c| \leq 2k$.}
    \label{fig:acsmall}
\end{center}
\end{figure}
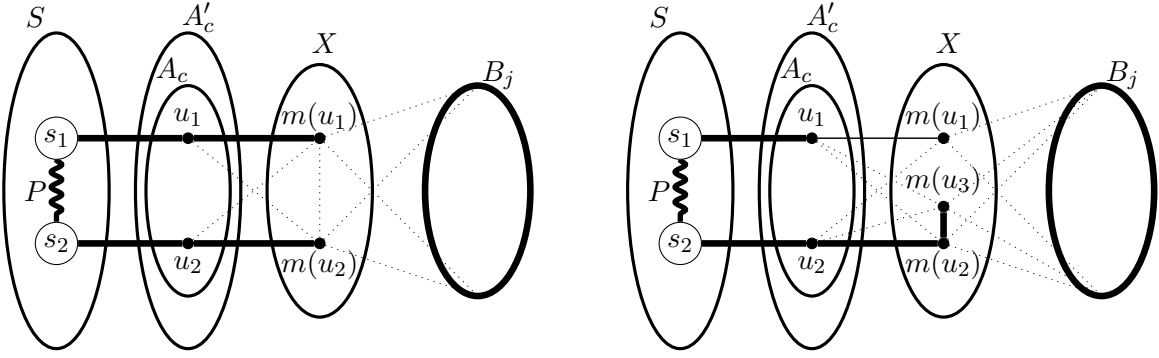

\begin{claim} \label{claim:aissmall}
    For each $c \in [k]$, we have that $|A_c| \leq 2k$. 
\end{claim}
\begin{proof}
Let $c \in [k]${, and recall the definitions of $M$ and $W$ given in Definition \ref{defs:QSoperators}}. Thus for each vertex $m \in M \setminus W$ there exists an index $c = ind(m)$ such that either $m \in B_c$, or $m$ is anticomplete to $B_c$ and $|B_c| = r$. Note that by the minimality of $A_c$ in $A'_c$, each vertex $v \in A_c$ has a unique corresponding neighbour $m(v)$ in $M\setminus W$ with the property that $N(m(v)) \cap A_c = \{v\}$. Furthermore, since $\bigcup_{c \in [k]} A_c \subseteq V(C)$ is anticomplete to $\bigcup_{c \in [k]} B_c \subseteq V(OPT \setminus C)$, it follows that $m(v) \not\in \bigcup_{c \in [k]} B_c$.

Suppose for a contradiction that $|A_c| \geq 2k+1$.  {By the pigeonhole principle, there exists an index $j \in [k]$ and vertices $u_1,u_2$, and $u_3$ in $A_c$ such that $\{m(u_1), m(u_2),m(u_3)\} \subseteq M \setminus W$ and $ind(m(u_1)) = ind(m(u_2)) = ind(m(u_3)) =  j$ and $|B_j|=r$.} It follows that  $B_j$ is anticomplete to $\{u_1,u_2, u_3\} \cup S$ since $\{u_1, u_2, u_3\} \cup S  \subseteq V(C)$, but each vertex in $B_j$ lies in a component of $OPT$ distinct from $C$. We note further that $S$ is anticomplete to $\{m(u_1),m(u_2), m(u_3)\}$ since $\{m(u_1),m(u_2), m(u_3)\} \subseteq M$ and $M = V(G) \setminus N\langle S \rangle$. 
    
First suppose that $\{m(u_1),m(u_2), m(u_3)\}$ does not induce a triangle in $G$; without loss of generality, we assume $m(u_1)$ and $m(u_2)$ are non-adjacent (see Figure \ref{fig:acsmall}, left). 
    Let $s_1$ and $s_2$ be neighbours of $u_1$ and $u_2$, respectively, in $S$ chosen to minimise $dist_{S}(s_1,s_2)$ (note that it is possible that $s_1 = s_2$). Let $P$ be a shortest $(s_1,s_2)$-path in $S$. Note that $P$ exists, since $S$ is connected. Then $m(u_1),u_1,P,u_2,m(u_2)$ together with $B_j$ contains an induced copy of $(P_5 + rK_1)$, contradicting that $G \in \mathcal{G}_r$. 
    
    We may assume instead that $\{m(u_1), m(u_2), m(u_3)\}$ induces a triangle in $G$; therefore $m(u_3)m(u_2) \in E(G)$ (see Figure \ref{fig:acsmall}, right). As above, let $s_1$ and $s_2$ be neighbours of $u_1$ and $u_2$, respectively, in $S$ chosen to minimise $dist_{S}(s_1,s_2)$ (again, it is possible that $s_1 = s_2$). Let $P$ be a shortest $(s_1,s_2)$-path in $S$. Note that $P$ exists, since $S$ is connected. In this case,  $u_1,P,u_2,m(u_2),m(u_3)$ together with $B_j$ contains an induced copy of $(P_5 + rK_1)$, again contradicting that $G \in \mathcal{G}_r$. {Hence we conclude instead that $|A_c| \le 2k$.}
\end{proof} 

{The following result will be used in the proof of Claim \ref{claim:Y'deffocanvas}.}

\begin{claim}\label{claim:y'zr}
     Let $i, \ell \in [t]$ with $i < \ell$ and $c \in [k]$. If $Y_{i,c,\ell} \not = \emptyset$, then $|Z_{\ell,c}| = r$ and $N(Y_{i,c,\ell}) \cap Z_{\ell,c} = \emptyset.$
\end{claim}
\begin{proof}
    Since $Y_{i,c,\ell}$ is non-empty, it follows from the definition of $Y_{i,c,\ell}$ that $|Z_{\ell,c}| = r$. Vertices $v$ in $Y_{i,c,\ell} \cup Z_{\ell,c}$ all satisfy $f(v) = c$ by definition, and thus each $y \in Z_{\ell,c}$ has no neighbours in $Y_{i,c,\ell}$, as desired. 
\end{proof}

\begin{claim}\label{claim:Y'deffocanvas}
    For each set $Y_{i,c,\ell}$, we have that $|Y_{i,c,\ell}| \leq 2$. 
\end{claim}

\begin{proof}    Suppose not. Let $Y_{i,c,\ell}$ be a counterexample, and let $y_1, y_2,y_3$ be distinct vertices in $Y_{i,c,\ell}$ (see Figure \ref{fig:caseslemsolwithallinc}). By Claim \ref{claim:y'zr}, $y_1$, $y_2,$ and $y_3$ have no neighbours in $Z_{\ell,c}$, and $|Z_{\ell,c}| = r$. By the minimality of $Y_{i,c,\ell}$, each vertex $y \in Y_{i,c,\ell}$ has a corresponding neighbour $n_\ell(y) \in (X_\ell \setminus N(Z_{\ell,c}))$ such that the only neighbour of $n_\ell(y)$ in $Y_{i,c,\ell}$ is $y$. Moreover, recall that $Y_{i,c,\ell}$ is an independent set, since all vertices $v$ in $Y_{i,c,\ell}$ satisfy $f(v) = c$.
    Note that since $i < \ell$, it follows that $v_i$ is anticomplete to $X_\ell$.
    The structure of the graph is depicted in Figure \ref{fig:caseslemsolwithallinc} for the case $i = 1$ and $\ell = 2$.
    
    If $n_\ell(y_1)$ and $n_\ell(y_2)$ are non-adjacent, then $n_\ell(y_1),y_1,v_i,y_2,n_\ell(y_2)$ together with $Z_{\ell,c}$ form a copy of $(P_5 + rK_1)$, a contradiction. Thus we may assume by symmetry that $\{n_\ell(y_1), n_\ell(y_2), n_\ell(y_3)\}$ form a clique. But then $n_\ell(y_3),n_\ell(y_1),y_1,v_i,y_2$ together with $Z_{\ell,c}$ form a copy of $(P_5 + rK_1)$, again a contradiction. 
\end{proof}

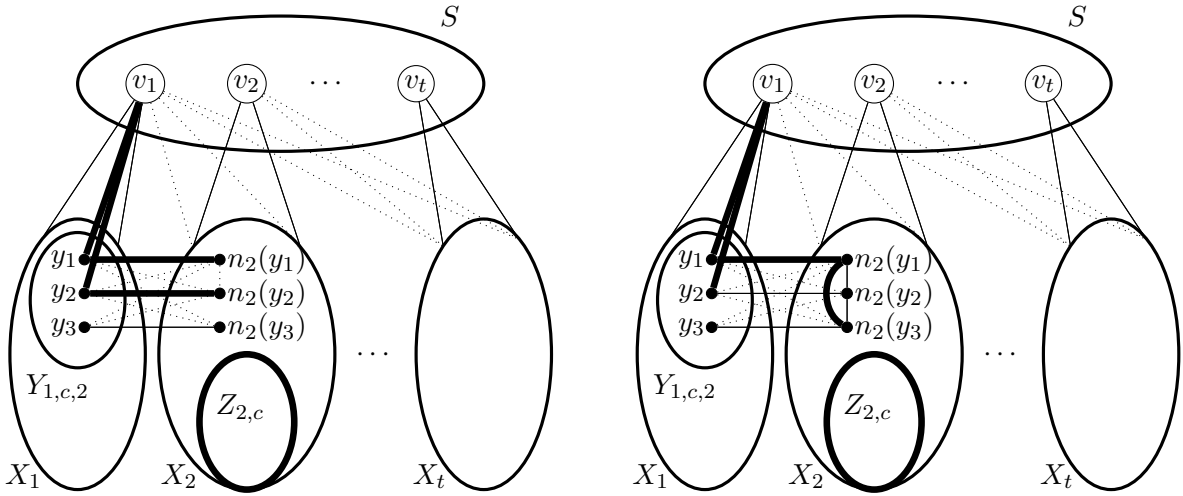
\begin{figure}[ht]
\tikzset{black/.style={shape=circle,draw=black,fill=black,inner sep=1pt, minimum size=4pt}}
\tikzset{white/.style={shape=circle,draw=black,fill=white,inner sep=1pt, minimum size=9pt}}
\tikzset{invisible/.style={shape=circle,draw=black,fill=black,inner sep=0pt, minimum size=0.1pt}}
\tikzset{decoration={snake,amplitude=.4mm,segment length=1mm,
                       post length=0mm,pre length=0mm}}
\begin{center}
\begin{tikzpicture}[scale=0.9]

\filldraw[color=black!100, fill=black!0, very thick] (0,2) ellipse (3 and 1); 
\node[] at (2.5,3) {$S$};

\filldraw[color=black!100, fill=black!0, very thick] (-3,-2) ellipse (1 and 2);
\node[] at (-3.8,-3.8) {$X_1$};

\filldraw[color=black!100, fill=black!0, very thick] (-3,-1.2) ellipse (0.7 and 1);
\node[] at (-3.3,-2.5) {$Y_{1,c,2}$};

\filldraw[color=black!100, fill=black!0, very thick] (-0.5,-2) ellipse (1.3 and 2); 
\node[] at (-1.5,-3.8) {$X_2$};
\filldraw[color=black!100, fill=black!0, line width=2.5] (-0.5,-3) ellipse (0.7 and 1);
\node[] at (-0.6,-2.8) {$Z_{2,c}$};

\filldraw[color=black!100, fill=black!0, very thick] (3,-2) ellipse (1 and 2); 
\node[] at (2.2,-3.8) {$X_t$};

        \node[white] (v1)  at (-2,2) {$v_1$};
        \node[white] (v2) at (-0.5,2) {$v_2$};
        \node[] at (0.7,2) {$\cdots$};
        \node[white] (vt) at (2,2) {$v_t$};

        \node[black] (y1) at (-2.9,-0.6) {};
        \node[] at (-3.2,-0.6) {$y_1$};
        \node[black] (y2) at (-2.9,-1.1) {};
        \node[] at (-3.2,-1.1) {$y_2$};
        \node[black] (y3) at (-2.9,-1.6) {};
        \node[] at (-3.2,-1.6) {$y_3$};
        \node[black] (n2y1) at (-0.9,-0.6) {};
        \node[] at (-0.2,-0.6) {$n_2(y_1)$};
        \node[black] (n2y2) at (-0.9,-1.1) {};
        \node[] at (-0.2,-1.1) {$n_2(y_2)$};
        \node[black] (n2y3) at (-0.9,-1.6) {};
        \node[] at (-0.2,-1.6) {$n_2(y_3)$};

        \node[invisible] (x11) at (-3.5,-0.25) {};
        \node[invisible] (x12) at (-2.4,-0.38) {};
        \node[invisible] (x21) at (-1.3,-0.43) {};
        \node[invisible] (x22) at (0.28,-0.38) {}; 
        \node[] at (1.4,-2) {$\cdots$};
        \node[invisible] (xt1) at (2.4,-0.4) {};
        \node[invisible] (xt2) at (3.52,-0.3) {};

        \draw[black] (v1)--(x11); 
        \draw[black] (v1)--(x12); 
        \draw[dotted] (v1)--(x21); 
        \draw[dotted] (v1)--(x22);
        \draw[dotted] (v1)--(xt1); 
        \draw[dotted] (v1)--(xt2);
        \draw[black] (v2)--(x21); 
        \draw[black] (v2)--(x22); 
        \draw[dotted] (v2)--(xt1); 
        \draw[dotted] (v2)--(xt2);
        \draw[black] (vt)--(xt1); 
        \draw[black] (vt)--(xt2);
        
        \draw[black] (y1)--(n2y1); 
        \draw[dotted] (y1)--(n2y2); 
        \draw[dotted] (y1)--(n2y3); 
        \draw[dotted] (y2)--(n2y1); 
        \draw[black] (y2)--(n2y2); 
        \draw[dotted] (y2)--(n2y3); 
        \draw[dotted] (y3)--(n2y1); 
        \draw[dotted] (y3)--(n2y2); 
        \draw[black] (y3)--(n2y3); 

        \draw[dotted] (n2y1)--(n2y2);
        \draw[line width=2.5] (y1) to (v1);
        \draw[line width=2.5] (v1) to (y2);
        \draw[line width=2.5] (y1) to (n2y1);
        \draw[line width=2.5] (y2) to (n2y2);

\end{tikzpicture}
\hskip 8mm
\begin{tikzpicture}[scale=0.9]

\filldraw[color=black!100, fill=black!0, very thick] (0,2) ellipse (3 and 1); 
\node[] at (2.5,3) {$S$};

\filldraw[color=black!100, fill=black!0, very thick] (-3,-2) ellipse (1 and 2);
\node[] at (-3.8,-3.8) {$X_1$};

\filldraw[color=black!100, fill=black!0, very thick] (-3,-1.2) ellipse (0.7 and 1);
\node[] at (-3.3,-2.5) {$Y_{1,c,2}$};

\filldraw[color=black!100, fill=black!0, very thick] (-0.5,-2) ellipse (1.3 and 2); 
\node[] at (-1.5,-3.8) {$X_2$};
\filldraw[color=black!100, fill=black!0, line width=2.5] (-0.5,-3) ellipse (0.7 and 1);
\node[] at (-0.6,-2.8) {$Z_{2,c}$};

\filldraw[color=black!100, fill=black!0, very thick] (3,-2) ellipse (1 and 2); 
\node[] at (2.2,-3.8) {$X_t$};

        \node[white] (v1)  at (-2,2) {$v_1$};
        \node[white] (v2) at (-0.5,2) {$v_2$};
        \node[] at (0.7,2) {$\cdots$};
        \node[white] (vt) at (2,2) {$v_t$};

        \node[black] (y1) at (-2.9,-0.6) {};
        \node[] at (-3.2,-0.6) {$y_1$};
        \node[black] (y2) at (-2.9,-1.1) {};
        \node[] at (-3.2,-1.1) {$y_2$};
        \node[black] (y3) at (-2.9,-1.6) {};
        \node[] at (-3.2,-1.6) {$y_3$};
        \node[black] (n2y1) at (-0.9,-0.6) {};
        \node[] at (-0.2,-0.6) {$n_2(y_1)$};
        \node[black] (n2y2) at (-0.9,-1.1) {};
        \node[] at (-0.2,-1.1) {$n_2(y_2)$};
        \node[black] (n2y3) at (-0.9,-1.6) {};
        \node[] at (-0.2,-1.6) {$n_2(y_3)$};

        \node[invisible] (x11) at (-3.5,-0.25) {};
        \node[invisible] (x12) at (-2.4,-0.38) {};
        \node[invisible] (x21) at (-1.3,-0.43) {};
        \node[invisible] (x22) at (0.28,-0.38) {}; 
        \node[] at (1.4,-2) {$\cdots$};
        \node[invisible] (xt1) at (2.4,-0.4) {};
        \node[invisible] (xt2) at (3.52,-0.3) {};

        \draw[black] (v1)--(x11); 
        \draw[black] (v1)--(x12); 
        \draw[dotted] (v1)--(x21); 
        \draw[dotted] (v1)--(x22);
        \draw[dotted] (v1)--(xt1); 
        \draw[dotted] (v1)--(xt2);
        \draw[black] (v2)--(x21); 
        \draw[black] (v2)--(x22); 
        \draw[dotted] (v2)--(xt1); 
        \draw[dotted] (v2)--(xt2);
        \draw[black] (vt)--(xt1); 
        \draw[black] (vt)--(xt2);
        
        \draw[black] (y1)--(n2y1); 
        \draw[dotted] (y1)--(n2y2); 
        \draw[dotted] (y1)--(n2y3); 
        \draw[dotted] (y2)--(n2y1); 
        \draw[black] (y2)--(n2y2); 
        \draw[dotted] (y2)--(n2y3); 
        \draw[dotted] (y3)--(n2y1); 
        \draw[dotted] (y3)--(n2y2); 
        \draw[black] (y3)--(n2y3); 

        \draw[black] (n2y1)--(n2y2);
        \draw[black] (n2y2)--(n2y3);
        \draw[] (n2y1) to[in=150, out=210] (n2y3);
        \draw[line width=2.5] (y1) to (v1);
        \draw[line width=2.5] (v1) to (y2);
        \draw[line width=2.5] (y1) to (n2y1);
        \draw[line width=2.5] (n2y1) to[in=150, out=210] (n2y3);

\end{tikzpicture}
\caption{The structure described in Lemma \ref{lem:solwithallinC}, Claim \ref{claim:Y'deffocanvas} with (for illustrative purposes) $\ell = 2$ and $i = 1$. Dashed lines indicate non-adjacency. To keep the image uncluttered, the dashed lines between $Z_{2,c}$ and each of $\{y_1,y_2,y_3\}$ and $\{n_2(y_1), n_2(y_2), n_2(y_3)\}$ have been omitted. On the left: the first case covered in Claim \ref{claim:Y'deffocanvas}, where $n_2(y_1)$ and $n_2(y_2)$ are non-adjacent: here $n_2(y_1),y_1,v_1,y_2,n_2(y_2)$ and $Z_{2,c}$ form an induced copy of $(P_5+rK_1)$ shown. On the right, the second case covered in Claim \ref{claim:Y'deffocanvas}: here, we assume $\{n_2(y_1),n_2(y_2),n_2(y_3)\}$ form a clique, and so $n_2(y_3),n_2(y_1),y_1,v_1,y_2$ and $Z_{2,c}$ form an induced copy of $(P_5 + rK_1)$. Since both cases lead to a contradiction, we conclude that $|Y_{1,c,2}| \leq 2$.}
    \label{fig:caseslemsolwithallinc}
\end{center}
\end{figure}

Thus $Q$ is indeed a canvas. 

\begin{claim} \label{claim:c_associated}
    $C$ is associated with $Q$. 
\end{claim}
\begin{proof}
    We know from the definition of $Q$ that $S$ dominates $C$ and that $f$ colours $G[S]$ correctly, so conditions (i) - (iii) of Definition~\ref{def:associated} are satisfied. Let $i \in [t]$ and $v \in V(C) \cap X_i$; let $c = f(v)$. {Recall that for all $c' \in [k],$ the vertices of $B_{c'}$ are coloured by $f$ and are hence vertices in the graph $OPT \setminus C$. Since $v \in V(C)$, it is anticomplete to $\bigcup_{c' \in [k]} B_{c'}$ {(condition (iv)(a)).} Given $c' \in [k]$, if $v \in A_{c'}, Z_{i,c'},$ or $Y_{i,c',\ell}$ for some $\ell \in [t]$ with $\ell > i$, then by the definition of $Q$ we have that $c' = c$ {(condition (iv)(b)).} Otherwise, assume $v$ is not in $A_c, Z_{i,c}, $ or $Y_{i,c, \ell}$ for any $\ell \in [t]$ with $\ell > i$. We check the five sets of Definition \ref{def:associated} in order.
    
    First, since $f$ is a proper colouring, $f(v)$ is not an element of $\{f(v_j) : j \in [t] \text{ and } vv_j \in E(G) \}$.  Second, from the definition of $A_c$, we have $$(N(v) \cap M) \setminus W \subseteq N(A_c') \cap (M \setminus W) = N(A_c) \cap (M \setminus W) \subseteq N(A_c).$$ Hence $f(v)$ is not an element of $\{c' \in [k]: N(v) \cap M \setminus W \not \subseteq N(A_{c'})\}$. Third, since $f(v) = c$, it follows that $v$ has no neighbours in $A_c \cup B_c \cup \bigcup_{\ell \in [t]}Z_{\ell, c} \cup \bigcup_{\substack{\ell, j \in [t], \ell < j}}Y_{\ell, c, j}$. Fourth, since $v$ is a vertex of $X_i$ coloured $c$ but $v \not \in Z_{i,c}$, it follows that $|Z_{i,c}| = r$ by definition of $Z_{i,c}$. Fifth, by definition of $Y'_{i,c}, $ we have $v \in Y'_{i,c}$. Let $w \in N(v) \cap X_j$ for some $j \in [t]$ with $j > i$. By definition of $Y_{i,c,j}$, it follows that if $|Z_{j,c}| = r$, then $w \in N(Y_{i,c,j})$ or $w \in N(Z_{j,c})$. Hence $f(v)$ is not an element of $\{c': \exists j > i \text{ such that } |Z_{j,c'}| = r \text{ and } (N(v) \cap X_j)\not\subseteq (N(Y_{i, c', j})\cap X_j) \setminus N(Z_{j, c'})\}$. Therefore, $f(v)$ {also} adheres to the restrictions of {condition (iv)(c) of}} Definition \ref{def:associated}, and as such $C$ is associated with $Q$.  
    \end{proof}

Let $C'$ be the output of Algorithm \ref{refining_algorithm} with inputs $G, Q$ and $L$. By Lemma \ref{lem:solqvalid}, $C'$ is indeed an $L$-colourable induced subgraph of $G$. {Furthermore, $C'$ is excellent: since $C'$ is produced from Algorithm \ref{refining_algorithm}, we have that $C' \in \mathcal{C}$, and the sets $A_c$ and $Y_{i,c,l}$ in $Q$ are chosen by definition to afford $C'$ the additional properties needed to be excellent.}

We now show that $(OPT\setminus C) \cup C'$ is $L$-colourable and $\omega((OPT \setminus C) \cup C') \ge \omega(OPT)$. To that end, we prove the following.

\begin{claim}\label{claim:c'isacompt}
    $C'$ and $OPT \setminus C$ are anticomplete to one another. 
\end{claim}
\begin{proof}
    Suppose for a contradiction that $v \in V(OPT) \setminus V(C)$ is adjacent to some $v' \in V(C')$. Since $C$ is a component of $OPT$, it follows that $v' \not \in V(C)$. Since $S \subseteq V(C)$, it follows that $v' \in V(C') \setminus S$. From line \ref{line:15} of Algorithm \ref{refining_algorithm}, it follows that $v \not\in \bigcup_{c \in [k]} B_c$.
    
    First we show that $W \cap OPT = \emptyset$. Fix $c \in [k]$ and $w \in W \cap OPT$ with $f(w) = c$ and $w \not \in B_c$. If $|B_c| = r$, then since all vertices in $W$ have a neighbour in $B_c$, $w$ cannot be coloured $c$. If $|B_c| < r$, then all vertices in $OPT$ of colour $c$ are in $B_c$, so $w \in B_c$, a contradiction. Hence $w$ cannot be coloured $c$ for any $c \in [k]$, so $W \cap OPT = \emptyset$. This implies that $v \in M \setminus W$. 
    
     Suppose that $v'$ has colour $c' \in [k]$ in a $k$-colouring of $C'$. It follows from line \ref{line:8} of Algorithm \ref{refining_algorithm} that $N(v') \cap (M \setminus W) \subseteq N(A_{c'})$.  Since $v \in N(v') \cap M$, this implies $v \in W \cup N(A_{c'})$. From the previous paragraph, it follows that $v \in N(A_{c'}) \cap M$. Since $A_{c'} \subseteq C$, it follows that $N(A_{c'}) \cap M \cap OPT = \emptyset$, a contradiction. 
\end{proof}

Claim \ref{claim:c'isacompt} implies that our $L$-colouring of $OPT \setminus C$ and our $L$-colouring of $C'$ can be combined to produce an $L$-colouring of $(OPT \setminus C) \cup C'$. 

The subgraph $C$ is associated with $Q$ by Claim \ref{claim:c_associated}, so Lemma \ref{lem:c'asgoodasc} implies that $\omega(C') \ge \omega(C)$. It follows that $(OPT \setminus C) \cup C'$ has weight at least as high as $OPT$. Thus $(OPT \setminus C) \cup C'$ is an optimum solution to the WML$k$CIS problem using strictly more excellent components, a contradiction.  Therefore, a solution to the WML$k$CIS  problem can be produced entirely from components of $\mathcal{C}$. 
\end{proof}

Having established all the necessary tools, we prove Theorem \ref{thm:mainthmp5} below.

\begin{proof}[Proof of Theorem \ref{thm:mainthmp5}.] 
    Fix $r \in \mathbb{N}$. We proceed by induction on $k$, the number of colours. By Lemma \ref{lem:stable}, a solution to the WM$(1)$CIS problem (which is the {\sc Maximum-Weight Independent Set} problem) for input graphs in $\mathcal{G}_r$ can be found in $\textnormal{poly}(v(G))$-time. 
    
    Now let $k \geq 2$, and suppose that WM$(k-1)$CIS is solvable in $\textnormal{poly}(v(G))$-time.  Let $G \in \mathcal{G}_r$. Let $L:V(G) \rightarrow 2^{[k]}$ be a list assignment for $G$. Note that we may assume $G$ is connected; if $G$ is disconnected, we apply this argument to each component of $G$ and take the union of the solutions for each component.  Let $\mathfrak{Q}_L$ be the set of all possible canvases with vertices in $G$ and list assignment $L$. By Lemma \ref{lem:Csmall}, $|\mathfrak{Q}_L| \in \textnormal{poly}(v(G))$. Let $\mathcal{C}$ be the set of outputs of Algorithm \ref{refining_algorithm} with inputs $G$ and $Q \in \mathfrak{Q}_L$. By Lemma \ref{lemma:alg_terminates}, since WM$(k-1)$CIS is solvable in $\textnormal{poly}(v(G))$-time, it follows that $\mathcal{C}$ can be found in $\textnormal{poly}(v(G))$-time. By Lemma \ref{lem:solwithallinC}, there exists a solution $OPT$ to the WML$k$CIS problem for $G$ where all components of $OPT$ are in $\mathcal{C}$. By Lemma \ref{lem:givenCcanfindOPT}, such a solution can be found in $\textnormal{poly}(v(G))$ time, as desired.
\end{proof}

\bibliographystyle{siam}
\bibliography{bibliog}

\end{document}